\documentclass[10pt]{article}

\usepackage[dvips]{color}
\usepackage{epsfig}
\usepackage{float,amsthm,amssymb,amsfonts}
\usepackage{ amssymb,amsmath,graphicx, amsfonts, latexsym}
\begin{document}
\theoremstyle{plain}
\newtheorem{theorem}{{\bf Theorem}}[section]
\newtheorem{corollary}[theorem]{Corollary}
\newtheorem{lemma}[theorem]{Lemma}
\newtheorem{proposition}[theorem]{Proposition}
\newtheorem{remark}[theorem]{Remark}

\theoremstyle{definition}
\newtheorem{defn}{Definition}
\newtheorem{definition}[theorem]{Definition}
\newtheorem{example}[theorem]{Example}
\newtheorem{conjecture}[theorem]{Conjecture}
\def\im{\mathop{\rm Im}\nolimits}
\def\dom{\mathop{\rm Dom}\nolimits}
\def\rank{\mathop{\rm rank}\nolimits}
\def\ker{\mathop{\rm ker}\nolimits}
\def\nullset{\mbox{\O}}

\def\implies{\; \Longrightarrow \;}

\def\GR{{\cal R}}
\def\GL{{\cal L}}
\def\GH{{\cal H}}
\def\GD{{\cal D}}
\def\GJ{{\cal J}}

\def\set#1{\{ #1\} }
\def\z{\set{0}}
\def\Sing{{\rm Sing}_n}
\def\nullset{\mbox{\O}}

\title{On certain semigroups of partial
contractions of a finite chain }
\author{\bf Abdullahi Umar and M. M. Zubairu \footnote{Corresponding Author. ~~Email: $aumar@pi.ac.ae$} \\[3mm]
\it\small Department of Mathematics, The Petroleum Institute, Sas Nakhl,\\
\it\small Khalifa University of Science and Technology, P. O. Box 2533 Abu Dhabi, UAE\\
\it\small  \texttt{aumar@pi.ac.ae}\\[3mm]
\it\small  Department of Mathematics, Bayero  University Kano, P. M. B. 3011 Kano Nigeria\\
\it\small  \texttt{mmzubairu.mth@buk.edu.ng}\\
}
\date{\today}
\maketitle\

\begin{abstract}
 Let $[n]=\{1,2,\ldots,n\}$ be a finite chain and let  $\mathcal{P}_{n}$  be the semigroup of partial transformations on $[n]$. Let $\mathcal{CP}_{n}=\{\alpha\in \mathcal{P}_{n}: (for ~all~x,y\in \dom~\alpha)~|x\alpha-y\alpha|\leq|x-y|\}$ be the subsemigroup of partial contraction mappings on $[n]$.  We  have shown that the semigroup $\mathcal{CP}_{n}$  and some of its subsemigroups are nonregular left abundant semigroups for all $n$ but not right abundant for $n\geq 4$.
 \end{abstract}

\section{Introduction and Preliminaries}
 Let $[n]=\{1,2, \ldots ,n\}$ be a finite chain, a map $\alpha$ which has domain and image both subsets of $[n]$ is said to be a \emph{transformation}.  A transformation $\alpha$ whose domain is a subset of $[n]$  (i. e., $\dom~\alpha \subseteq [n]$) is said to be  \emph{partial}. The collection of all partial transformations of $[n]$ is known as semigroup of partial transformations, usually denoted by $\mathcal{P}_{n}$. A map $\alpha\in \mathcal{P}_{n}$ is said to be \emph{order preserving} (resp., \emph{order reversing}) if  (for all $x,y \in \dom~\alpha$) $x\leq y$ implies $x\alpha\leq y\alpha$ (resp. $x\alpha\geq y\alpha$); is \emph{order decreasing} if (for all $x\in \dom~\alpha$) $x\alpha\leq x$;  an \emph{isometry} (i. e., \emph{ distance preserving}) if (for all $x,y \in \dom~\alpha$) $|x\alpha-y\alpha|=|x-y|$;   a \emph{contraction} if (for all $x,y \in \dom~\alpha$) $|x\alpha-y\alpha|\leq |x-y|$. Let $\mathcal{CP}_{n}=\{\alpha\in \mathcal{P}_{n}:(for ~all~x,y\in \dom~\alpha)~ |x\alpha-y\alpha|\leq |x-y|\}$ and $\mathcal{OCP}_{n}=\{\alpha\in \mathcal{CP}_{n}: (for ~all~x,y\in \dom~\alpha)~x\leq y ~ implies ~ x\alpha\leq y\alpha\}$ be the subsemigroups of \emph{partial contractions} and of \emph{order preserving partial contractions} of $[n]$, respectively.

  Further, the collection of all order preserving or order reversing partial contractions denoted by $\mathcal{ORCP}_{n}$ is a subsemigroup of $\mathcal{ORP}_{n}$ (where $\mathcal{{ORP}}_{n}$ denotes the semigroup of order preserving or order reversing partial transformations of $[n]$).

     In 2013, Zhao and Yang \cite{py} characterized the Green's relations on the subsemigroup $\mathcal{OCP}_{n}$ of $\mathcal{CP}_{n}$ ($\mathcal{OCP}_{n}$ is the semigroup of order preserving partial contractions on $[n]$). Recently, Ali \emph{et al.} \cite{mmz} obtained a necessary and sufficient condition for an element in $\mathcal{CP}_{n}$ to be regular and also described all its Green's equivalences. Most of the results concerning regularity and Green's relations for some subsemigroups of $\mathcal{CP}_{n}$ can be deduced from the results obtained in that paper.
   Zhao and Yang \cite{py} have shown that the semigroup $\mathcal{OCP}_{n}$ ($n>2$) is nonregular.  Similarly, Ali \emph{et al.} \cite{mmz} have shown that the semigroups $\mathcal{CP}_{n}$ and  $\mathcal{ORCP}_{n}$ are nonregular for $n>2$. Thus, there is a need to identify the class of semigroups to which they belong, for example, whether they are \emph{abundant} semigroups. Therefore, this paper is a natural sequel to Ali \emph{et al.} \cite{mmz}.

 This section includes a brief introduction giving some basic definitions and introducing some new concepts. In section 2, we characterize all the  starred Green's relations on the semigroups $\mathcal{CP}_{n}$, $\mathcal{ORCP}_{n}$ and $\mathcal{OCP}_{n}$ and show that  $\mathcal{D}^{*}=\mathcal{J}^{*}$ \cite{FOUN2}. In section 3, we show that the collection of all \emph{strongly regular} elements of $\mathcal{ORCP}_{n}$ is a subsemigroup and  in section 4, we show that the semigroups $\mathcal{CP}_{n}$, $\mathcal{ORCP}_{n}$ and $\mathcal{OCP}_{n}$ are left abundant for all $n$ but not right abundant for $n\geq 4$.

 For standard concepts in semigroup theory, we refer the reader to Howie \cite{Howie1} and Higgins \cite{ph}.

 Let $\alpha$ be element of $\mathcal{CP}_{n}$. Let  $\dom\ \alpha$, $\im~\alpha$,  $h~(\alpha)$ and $F(\alpha)$ denote, the \emph{domain of} $\alpha$, \emph{image of} $\alpha$, $|\im~\alpha|$  and $\{x\in\dom ~\alpha: x\alpha=x\}$ (i. e., the \emph{set of fixed points of} $\alpha$), respectively. For $\alpha,\beta \in \mathcal{CP}_{n}$,  the composition of $\alpha$ and $\beta$ is defined as $x(\alpha \circ \beta) =((x)\alpha)\beta$ for any $x$ in $\dom~\alpha$.  Without ambiguity, we shall be using the notation $\alpha\beta$ to denote $\alpha \circ\beta$.

   Next, given any transformation $\alpha$ in $\mathcal{ORCP}_{n}$, the domain of $\alpha$  is partitioned into $p-blocks$ by the relation $\ker~\alpha=\{(x,y)\in [n]\times [n]: x\alpha=y\alpha\}$ and so as in \cite{HRS}, $\alpha$ can be expressed as

    \begin{equation}\label{1}
    \alpha=\left( \begin{array}{cccc}
                           A_{1} & A_{2} & \ldots & A_{p} \\
                           x_{1} & x_{2} & \ldots & x_{p}
                         \end{array}
   \right)\in \mathcal{P}_{n}  ~~  (1\leq p\leq n),
    \end{equation}

    \noindent where, $A_{i}$ ($1\leq i\leq p$) are equivalence classes under the relation $\ker~\alpha$, i. e.,  $A_{i}=x_{i}\alpha^{-1}$  ($1\leq i\leq p$) and further $\textbf{Ker}~\alpha$ is ordered under the usual ordering, i. e.,  $\textbf{Ker}~\alpha=\{A_{1}< A_{2}< \ldots < A_{p}\}$. Thus for the rest of the content of the paper we shall consider $\alpha$ to be as expressed in \eqref{1} unless otherwise specified.

      Now, let $\textbf{Ker}~\alpha=\left(A_{i}\right)_{i\in[p]}=\{A_{1}, A_{2},\ldots, A_{p}\}$ be the partition of $\dom~\alpha$. A subset $T_{\alpha}$ of $[n]$ is said to be a \emph{transversal} of the partition $\textbf{Ker}~\alpha$ if $|A_{i}\cap T_{\alpha}|=1$ ($1\leq i\leq p$). A transversal  $T_{\alpha}$  is said to be \emph{relatively convex} if for all $x,y\in T_{\alpha}$ with $x\leq y$ and if $x\leq z\leq y$ ($z\in \dom~\alpha$), then $z\in T_{\alpha}$. Notice that every convex transversal is necessarily relatively convex but not \emph{vice-versa}. A transversal $T_{\alpha}$ is said to be \emph{admissible}  if and only if the map $A_{i}\mapsto t_{i}$  ($t_{i}\in T_{\alpha},\, i\in\{1,2,\ldots,p\}$) is a contraction, see \cite{mmz}. Notice that every convex transversal is admissible but not \emph{vice-versa}.

Let $S$ be a semigroup and $A$ be a subset of $S$, $\langle A\rangle$ denotes the semigroup generated by $A$ and is a subsemigroup of $S$. If $\langle A\rangle=S$ then $A$ is said to generate $S$ and also  $\langle A\rangle=A$ if and only if  $A$ is a subsemigroup of $S$. An element $a$ in a semigroup $S$ is said to be an idempotent if and only if $a^{2}=a$. As usual $E(S)$ denotes the set of all idempotents in $S$. It is well known that an element $\alpha$ in $\mathcal{P}_{n}$ is idempotent if and only if $\im~\alpha=F(\alpha)$. (Equivalently, $\alpha$ is idempotent if and only if $x_{i}\in A_{i}$ for  $1\leq i\leq p$, that is to say the blocks in $\textbf{Ker}~\alpha$ are \emph{stationary}.)

 Next, we quote some basic lemmas  from \cite{mmz} which would be useful for some of the subsequent results.

\begin{lemma}[\cite{mmz}, Lemma 1.3]\label{p1} For $n\geq 4$, let $\alpha\in \mathcal{CP}_{n}$   such that there exists $k\in\{2,\ldots,p-1\}$ ($3\leq p\leq n $)  with $|A_{k}|\geq 2$.  If $A_{i}<A_{j}$ ($i<j$) for all $i,j \in  \{1,2,\ldots,p\}$  then the partition $\textbf{Ker}~\alpha=\{ A_{1}, A_{2},\ldots, A_{p}\}$ of $\dom~\alpha$ has no relatively convex transversal.
 \end{lemma}

 \begin{corollary}[\cite{mmz}, Lemma 1.4]\label{p2} For $n\geq 4$, let $\alpha\in \mathcal{ORCP}_{n}$   such that there exists $k\in\{2,\ldots,p-1\}$ ($p\geq 3$) with $|A_{k}|\geq 2$. Then the partition $\textbf{Ker}~\alpha=\{ A_{1}, A_{2},\ldots, A_{p}\}$ of $\dom~\alpha$ has no relatively convex transversal.
\end{corollary}

\begin{lemma}[\cite{mmz}, Lemma 1.5]\label{p3}
 Let $\alpha\in \mathcal{CP}_{n}$   such that $A_{i}<A_{j}$ for all $i<j$ in $\{1,2,\ldots,p\}$  ($p\geq 3$).  If  $|A_{i}|=1$  for all $2\leq i\leq p-1$ then the partition $\textbf{Ker}~\alpha$ of $\dom~\alpha$ has an \emph{admissible}  transversal $T_{\alpha}$.
  \end{lemma}
\begin{lemma}[\cite{mmz} Lemma 1.8]\label{tofa} Let $\alpha\in \mathcal{CP}_{n}$ and let $A$ be a convex subset of $\dom~\alpha$. Then $A\alpha$ is convex.
 \end{lemma}

\begin{lemma}[\cite{mmz}, Corollary 4.2]\label{co2} Let $\alpha\in \mathcal{ORCP}_{n}$. If $|\im~\alpha|\geq 3$, then $\alpha$ is regular if and only if either $\min A_{p}-x_{p}=\max A_{1}-x_{1}=d$ and $A_{i}=\{x_{i}+d\}$ or  $\min A_{p}-x_{1}=\max A_{1}-x_{p}=d$ and $A_{i}=\{x_{p-i+1}+d\}$, for $i=2,\ldots,p-1$. Equivalently, if $T_{\alpha}$ is admissible then $\alpha$ is regular if and only if $\alpha|_{T_{\alpha}}$ is an isometry.
\end{lemma}

\section{Starred Green's relations}

 Let $S$ be a semigroup. A relation $\mathcal{L}^{*}$ defined as ($\forall ~a, ~b\in S$) $a \mathcal{L}^{*} b$ \emph{if and only if} $a$, $b$ are related by the Green's $\mathcal{L}^{*}$ relation in some oversemigroup of $S$, is known as the starred Green's $\mathcal{L}$ relation. The relation $\mathcal{R}^{*}$ is defined dually, while the relation $\mathcal{D}^{*}$ is defined as the join of the relations $\mathcal{L}^{*}$ and $\mathcal{R}^{*}$. The intersection of $\mathcal{L}^{*}$ and $\mathcal{R}^{*}$ is denoted by $\mathcal{H}^{*}$. A semigroup $S$ is said to be \emph{left abundant} (resp., \emph{right abundant}) if each $\mathcal{L}^{*}-$class  (resp., $\mathcal{R}^{*}-$class) contains an idempotent, it is called \emph{abundant} if each $\mathcal{L}^{*}-$class and $\mathcal{R}^{*}-$class of $S$ contains an idempotent. An abundant semigroup in which the set  $E(S)$, of  idempotents of $S$ is a subsemigroup of $S$ is called \emph{ quasi adequate} and if $E(S)$ is commutative then $S$ is called \emph{adequate} \cite{FOUN, FOUN2}.

 Many nonregular classes of transformation semigroups were shown to be either abundant or adequate, for example see \cite{ua, uaa,uaaa,uaaaa,PZ, SL,ZP,MS}. Recently, AlKharousi \emph{et al.} have shown that the semigroup $\mathcal{OCI}_{n}$, of all order preserving one to one contraction maps of a finite chain is adequate \cite{garbac2}. In this section we are going to show that the semigroups $\mathcal{CP}_{n}$, $\mathcal{OCP}_{n}$, $\mathcal{CT}_{n}$ and $\mathcal{OCT}_{n}$ are all left abundant (for all $n$) but not right abundant for $n\geq 4$.

We shall use the following notation from (\cite{Howie1}, Chapter 2). If $U$ is a subsemigroup of a semigroup $S$ then $a \mathcal{L}^{U}b$ means that there exist $u,v\in U^{1}$ such that $ua=b$ and $vb=a$, while    $a \mathcal{L}^{S}b$ means that there exist $x,y\in S^{1}$ such that $xa=b$ and $yb=a$. Similarly, for the relation $\mathcal{R}$.

  Some of the earlier results concerning starred Green's relations on a transformation semigroup were obtained by Umar \cite{ua, uaa, uaaa, uaaaa}, where he described all the starred relations on the semigroups of order decreasing full and of order decreasing partial one-one transformations of a chain,  these papers marked the beginning of the study of these relations on a transformation semigroup. Recently, Garba \emph{et al.} characterized these relations on the semigroup of full contraction maps and of order preserving full contraction maps of a finite chain: $\mathcal{CT}_{n}$ and $\mathcal{OCT}_{n}$, respectively \cite{garbac}. In this section, we characterize these relations on the more general semigroup of partial contractions $\mathcal{CP}_{n}$ and its subsemigroups of order preserving or order reversing partial contraction maps of a finite chain $\mathcal{ORCP}_{n}$, and of order preserving  partial contraction maps of a finite chain  $\mathcal{OCP}_{n}$, respectively. We equally show that the relations $\mathcal{D}^{*}$ and $\mathcal{J}^{*}$ coincide on these semigroups.

  To begin our investigation let us start with the following. The relations $\mathcal{L}^{*}$ and $\mathcal{R}^{*}$ have the following characterizations as described in (\cite{Howie1}, Exercise 2.6.7-9) or as described in \cite{FOUN}.
\begin{equation}\label{f1}
\mathcal{L}^{*}=\{(a,~b): (\forall x,y\in S^{1})~ax=ay\Leftrightarrow bx=by\}
\end{equation} and
\begin{equation}\label{f2}
\mathcal{R}^{*}=\{(a,~b): (\forall x,y\in S^{1})~xa=ya\Leftrightarrow xb=yb\}
\end{equation}

 We next give the characterizations of these relations on the  semigroups $\mathcal{CP}_{n}$, $\mathcal{ORCP}_{n}$ and $\mathcal{OCP}_{n}$ as follows: Let $S$ be a semigroup in $\{\mathcal{CP}_{n}, \mathcal{ORCP}_{n}, \mathcal{OCP}_{n}\}$.

 \begin{theorem}\label{starr} Let $\alpha, \beta\in S$. Then
 \begin{itemize}
   \item [(i)] $\alpha\mathcal{L}^{*}\beta$ if and only if $\im~\alpha=\im~\beta$.
   \item [(ii)] $\alpha\mathcal{R}^{*}\beta$ if and only if ${\ker}~\alpha={\ker}~\beta$.
   \item [(iii)] $\alpha\mathcal{H}^{*}\beta$ if and only if $\im~\alpha=\im~\beta$ and $\ker~\alpha=\ker~\beta$.
   \item [(iv)] $\alpha\mathcal{D}^{*}\beta$ if and only if $|\im~\alpha|=|\im~\beta|$.
 \end{itemize}
 \end{theorem}

 \begin{proof} (i) Let $\alpha, \beta$ be elements in $S\in\{\mathcal{CP}_{n}, \mathcal{ORCP}_{n}, \mathcal{OCP}_{n}\}$ such that $\alpha\mathcal{L}^{*}\beta$ and $\im~\alpha =\{ x_{1}, x_{2}, \ldots, x_{p}\}$. Further, let $\gamma=\left( \begin{array}{cccc}
                            x_{1} & x_{2} & \ldots & x_{p} \\
                            x_{1} & x_{2} & \ldots & x_{p}
                           \end{array}
\right)$.

Then clearly $\gamma\in S$ and \begin{align*}&~~~~~\alpha.\left( \begin{array}{cccc}
                            x_{1} & x_{2} & \ldots & x_{p} \\
                            x_{1} & x_{2} & \ldots & x_{p}
                           \end{array}
\right)=\alpha\cdot1_{[n]}~~~(by ~~\eqref{f1}) \\ &\Leftrightarrow \beta\cdot\left( \begin{array}{cccc}
                            x_{1} & x_{2} & \ldots & x_{p} \\
                            x_{1} & x_{2} & \ldots & x_{p}
                           \end{array}
\right)=\beta\cdot 1_{[n]}\end{align*}  which implies that $\im~\beta\subseteq \{ x_{1}, x_{2}, \ldots, x_{p}\}= \im~\alpha$. Similarly in the same manner we can have $\im~\alpha\subseteq \im~\beta$. Thus, $\im~\alpha=\im~\beta$.

Conversely, suppose that $\im~\alpha=\im~\beta$. Then by (\cite{Howie1}, Excercise 2.6, 17) $\alpha \mathcal{L}^{\mathcal{P}_{n}}\beta$ and it follows from definition that $\alpha\mathcal{L}^{*}\beta$. Thus, the result follows.

(ii) Suppose that $\alpha,\beta\in S$ and $\alpha\mathcal{R}^{*}\beta$. Then $(x,y)\in \ker~\alpha$ if and only if \begin{align*} x\alpha=y\alpha &\Leftrightarrow
\left(\begin{array}{c}
                                                                                                                                                \dom~\alpha \\
                                                                                                                                                x
                                                                                                                                              \end{array}
\right)\circ\alpha=\left(\begin{array}{c}
                                                                                                                                                \dom~\alpha\\
                                                                                                                                                y
                                                                                                                                              \end{array}
\right)\circ\alpha ~~~~(by~~\eqref{f2})
\\&\Leftrightarrow
\left(\begin{array}{c}
                                                                                                                                                \dom~\alpha \\
                                                                                                                                                x
                                                                                                                                              \end{array}
\right)\circ\beta=\left(\begin{array}{c}
                                                                                                                                                \dom~\alpha \\
                                                                                                                                                y
                                                                                                                                              \end{array}
\right)\circ\beta.\\&\Leftrightarrow x\beta=y\beta \\&\Leftrightarrow (x,~y)\in \ker~\beta.
\end{align*}
 Hence $\ker~\alpha=\ker~\beta$.

 Conversely, suppose that $\ker~\alpha=\ker~\beta$.  Then by (\cite{Howie1}, Excercise 2.6, 17) $\alpha \mathcal{R}^{\mathcal{P}_{n}}\beta$ and it follows from definition that $\alpha\mathcal{R}^{*}\beta$.
(iii) This follows from (i) and (ii).

(iv) Suppose that $\alpha\mathcal{D}^{*}\beta$. Then by (\cite{Howie1}, Proposition 1.5.11) there exist elements $\gamma_{1},~\gamma_{2}, \ldots,~\gamma_{2n-1}\in ~S$ such that $\alpha\mathcal{L}^{*}\gamma_{1}$, $\gamma_{1}\mathcal{R}^{*}\gamma_{2}$, $\gamma_{2}\mathcal{L}^{*}\gamma_{3},\ldots,$ $\gamma_{2n-1}\mathcal{R}^{*}\beta$ for some $n\in ~ \mathbb{{N}}$. Thus, by (i) and (ii) we have $\im~\alpha=\im~\gamma_{1}$, ${\ker}~\gamma_{1}={\ker}~\gamma_{2}$, $\im~\gamma_{2}=\im~\gamma_{3}\ldots,$ $\ker~\gamma_{2n-1}=\ker~\beta$. This implies that $|\im~\alpha|=|\im~\gamma_{1}|=|\dom~\gamma_{1}/ \ker~\gamma_{1}|=|\dom~\gamma_{2}/ \ker~\gamma_{2}|=\ldots=|\dom~\gamma_{2n-1}/ \ker~\gamma_{2n-1}|=|\dom~\beta/ \ker~\beta|=|\im~\beta|.$

Conversely, suppose that $|\im~\alpha|=|\im~\beta|$ where \begin{equation}\label{2} \alpha=\left(\begin{array}{cccc}
                                                                            A_{1} & A_{2} & \ldots & A_{p} \\
                                                                            x_{1} & x_{2} & \ldots & x_{p}
                                                                          \end{array}
\right)~~ and~~ \beta=\left(\begin{array}{cccc}
                                                                            B_{1} & B_{2} & \ldots & B_{p} \\
                                                                            y_{1} & y_{2} & \ldots & y_{p}
                                                                          \end{array}
\right)~~ (p\leq n)\end{equation}

\noindent such that the map $x_{i} \mapsto y_{i}$ ($i=1,2,\ldots,p$) is an isometry. Then the map $\gamma=\left(\begin{array}{cccc}
                                                                            B_{1} & B_{2} & \ldots & B_{p} \\
                                                                            x_{1} & x_{2} & \ldots & x_{p}
                                                                          \end{array}
\right)$ is in $S$. Moreover, by (i) and (ii), it follows that $\alpha\mathcal{L}^{*}\gamma$ and $\gamma\mathcal{R}^{*}\beta$. Thus, by (\cite{Howie1} Proposition 1.5.11) we have $\alpha\mathcal{D}^{*}\gamma$ and the proof is complete.
 \end{proof}

 A \emph{left (resp. right)} $*-$\emph{ideal} of a semigroup $S$ is defined as the \emph{left (rep. right)} ideal of $S$ for which $L^{*}_{a}\subseteq I~(resp.~R^{*}_{a}\subseteq I)$ for all $a\in I$. A subset $I$ of a semigroup $S$ is a $*-$ideal if it is both left and right $*-$ideal of $S$. The principal $\mathcal{J}^{*}$ $*-$ideal, $J^{*}(a)$, generated by $a\in S$ is the intersection of all $*-$ideal  of $S$ containing $a$, where  the relation $\mathcal{J}^{*}$ is defined as: $a\mathcal{J}^{*}b$ if and only if $J^{*}(a)=J^{*}(b)$ for all $a,~b\in S$. We now recognize the following lemma from Fountain \cite{FOUN2}.

 \begin{lemma}[ \cite{FOUN2}]\label{equl} Let $a$,$b$ be elements of a semigroup $S$. Then $b\in J^{*}(a)$ if and only if there are elements $a_{0},a_{1}, \ldots, a_{n}\in S$, $x_{1},x_{2}, \ldots, x_{n}, ~y_{1}, y_{2},\ldots, y_{n}\in S^{1}$ such that $a=a_{0}$, $b=a_{n}$ and $(a_{i}, x_{i}a_{i-1}y_{i})\in\mathcal{D}^{*}$ for $i=1,2,\ldots,n$.

 \end{lemma}

 As in \cite{uaa}, we immediately have:
 \begin{lemma}\label{equ22} Let $S$ be in $\{\mathcal{CP}_{n}, \mathcal{ORCP}_{n}, \mathcal{OCP}_{n}\}$. Then for $\alpha,~\beta\in S$, $\alpha\in J^{*}(\beta)$ implies $|\im~\alpha|\leq|\im~\beta|$.

 \end{lemma}
 \begin{proof}
 Let $\alpha\in J^{*}(\beta)$, then by lemma\eqref{equl}, there exist $\eta_{0}, \eta_{1}\ldots,\eta_{n}\in S$,\\ $\rho_{1},\rho_{2},\ldots,\rho_{n},~~\tau_{1}, \tau_{2},\ldots,\tau_{n}\in S^{1}$ such that $\beta=\eta_{0}$, $\alpha=\eta_{n}$ and $(\eta_{i},~~\rho_{i}\eta_{i-1}\tau_{i})\in \mathcal{D}^{*}$ for $i=1,2,\ldots,n$. Thus, by Theorem\eqref{starr}(iv), it implies that $$
 |\im~\eta_{i}|=|\im~\rho_{i}\eta_{i-1}\tau_{i}|\leq |\im~\eta_{i}|~~for~~i=1,2,\ldots,n,$$
which implies that $|\im~\alpha|\leq |\im~\beta|$.
 \end{proof}

 Notice that, $\mathcal{D}^{*}\subseteq \mathcal{J}^{*}$ and together with Lemma~\eqref{equ22} we have:

 \begin{corollary} On the semigroups $\mathcal{CP}_{n}$, $\mathcal{ORCP}_{n}$ or $\mathcal{OCP}_{n}$ we have $\mathcal{D}^{*}= \mathcal{J}^{*}$.

 \end{corollary}
   We now are going to show in the next lemma that if $S\in\{\mathcal{CP}_{n},~\mathcal{OCP}_{n},~\mathcal{ORCP}_{n}\}$ then $S$ is left abundant.

  \begin{lemma} Let $S\in\{\mathcal{CP}_{n}, \mathcal{OCP}_{n}, \mathcal{ORCP}_{n}\}$. Then $S$ is left abundant.
 \end{lemma}

 \begin{proof}
 Let $\alpha\in S$ and $L^{*}_{\alpha}$ be an $\mathcal{L}^{*}-class$ of $\alpha$ in $S$, where $\alpha=\left(\begin{array}{cccc}
                                                                            A_{1} & A_{2} & \ldots & A_{p} \\
                                                                            x_{1} & x_{2} & \ldots & x_{p}
                                                                          \end{array}
\right)$ ($1\leq p\leq n$). Define  $\gamma=\left(\begin{array}{cccc}
                                                                            x_{1} & x_{2} & \ldots & x_{p} \\
                                                                            x_{1} & x_{2} & \ldots & x_{p}
                                                                          \end{array}
\right).$ Clearly $\gamma^{2}=\gamma\in S$ and $\im~\alpha=\im~\gamma$, therefore by Theorem\eqref{starr}(i), $\alpha \mathcal{L}^{*}\gamma$, which means that $\gamma\in L^{*}_\alpha$. Thus, $S$ is left abundant, as required.

 \end{proof}

\begin{lemma}Let $S\in\{\mathcal{CP}_{n}, \mathcal{ORCP}_{n}, \mathcal{OCP}_{n}\}$. Then for $ n\geq 4$, $S$ is  not right abundant.
\end{lemma}

\begin{proof} Let $n=4$ and consider $\alpha=\left(\begin{array}{ccc}
                                                                            1 & \{2,3\} & 4  \\
                                                                            1 & 2 &  3
                                                                          \end{array}
\right)$. Clearly $\alpha$ is in $S$ and $R^{*}_\alpha=\left\{   \left(\begin{array}{ccc}
                                                                            1 & \{2,3\} & 4  \\
                                                                            1 & 2 &  3
                                                                          \end{array}
\right),\left(\begin{array}{ccc}
                                                                            1 & \{2,3\} & 4  \\
                                                                            3 & 2 &  1
                                                                          \end{array}
\right), \left(\begin{array}{ccc}
                                                                            1 & \{2,3\} & 4  \\
                                                                            2 & 3 &  4
                                                                          \end{array}
\right), \left(\begin{array}{ccc}
                                                                            1 & \{2,3\} & 4  \\
                                                                            4 & 3 &  2
                                                                          \end{array}
\right)   \right\}$, which has no idempotent element.
\end{proof}

 \begin{remark} Let $S\in\{\mathcal{CP}_{n}, \mathcal{ORCP}_{n}, \mathcal{OCP}_{n}\}$. Then for $1\leq n\leq 3$, $S$ is right abundant.
\end{remark}

\section{On strongly regular elements in $\mathcal{ORCP}_{n}$ }
 In 2018, Ali \emph{et al.} \cite{mmz} characterized the regular elements in $\mathcal{ORCP}_{n}$ and showed that the semigroup $\mathcal{ORCP}_{n}$ is not regular. We begin the section with a remark concerning the product of regular elements in $\mathcal{ORCP}_{n}$.

\begin{remark}
Product of regular elements in   $\mathcal{ORCP}_{n}$ or  $\mathcal{OCP}_{n}$ is not necessarily regular.  Consider $\alpha=\left(\begin{array}{cc}
                                       1 & 3   \\
                                       1 & 3
                                     \end{array}
 \right)$ and $\beta=\left(\begin{array}{cc}
                                       1 & \{2,3\}   \\
                                       1 & 2
                                     \end{array}
 \right)\in \mathcal{ORCP}_{3}$. Then $\alpha\beta=\left(\begin{array}{cc}
                                       1 & 3   \\
                                       1 & 2
                                     \end{array}
 \right)$ is not regular.
\end{remark}

Next, we recognize the following result (due to Hall (\cite{Hall}, Proposition 1.)) which is crucial in proving some of the results below.
\begin{proposition} (\cite{Hall}, Proposition 1.)\label{idp1}   Let $S$ be an arbitrary semigroup. Then the following are equivalent:
\begin{itemize}
\item[(i)] For all idempotents $e$ and $f$ of $S$, the element $ef$ is regular;
\item[(ii)] $Reg(S)$ is regular subsemigroup;
\item[(iii)]    $\langle E(S) \rangle$  is a regular semigroup.
\end{itemize}
\end{proposition}

 We now introduce the following definition. A regular element $\alpha$  in $\mathcal{ORCP}_{n}$  is said to be \emph{strongly regular} if and only if $\textbf{Ker}~\alpha$ has a convex transversal $T_{\alpha}$.
 For example, the contraction $\alpha=\left(\begin{array}{cc}
                                       1 & 3   \\
                                       3 & 1
                                     \end{array}
 \right)\in \mathcal{ORCP}_{3}$ is regular but not strongly regular, since $T_{\alpha}=\{1,3\}$ is not convex. On the other hand, the contraction $\beta=\left(\begin{array}{cc}
                                       1 & \{2,3\}   \\
                                       3 & 4
                                     \end{array}
 \right)\in \mathcal{ORCP}_{3}$ is strongly regular, since $T_{\beta}=\{1,2\}$ is convex.

  Now let $SReg(\mathcal{ORCP}_{n})$ denote the set of all strongly regular elements in $\mathcal{ORCP}_{n}$. Then we have the following two lemmas:

\begin{lemma}\label{mm} Let $\epsilon$ be an idempotent element in $SReg(\mathcal{ORCP}_{n})$.  Then $\alpha$ can be expressed as $\left(\begin{array}{cccccc}
                                                                            A_{1} & a+2 & a+3 & \ldots & a+p-1 & A_{p} \\
                                                                            a+1 & a+2 & a+3 & \ldots & a+p-1 & a+ p
                                                                          \end{array}
\right)$,  where $a+1=\max A_{1}$, $a+p=\min A_{p}$.
\end{lemma}

\begin{proof}
Let $\epsilon\in SReg(\mathcal{ORCP}_{n})$ be of height $p$. Then by the contrapositive of Lemma\eqref{p2} we see that  $\textbf{Ker}~\epsilon=\{A_{1}<\{a+2\}<\ldots<\{a+p-1\}<A_{p}\}$, and so $T_{\epsilon}$ is convex. Thus by Lemma \eqref{tofa},  $T_{\epsilon}\epsilon=\im ~\epsilon$ is convex and hence  $$\epsilon=\left(\begin{array}{cccccc}
                                                                            A_{1} & a+2 & a+3 & \ldots & a+p-1 & A_{p} \\
                                                                            x+1 & x+2 & x+3 & \ldots & x+p-1 & x+p
                                                                          \end{array}
\right).$$
\noindent However, since $\epsilon$ is an idempotent then the blocks of $\textbf{Ker}~\epsilon$ are stationary i. e., $x+1\in A_{1}$, $x+p\in A_{p}$, and $x+i=a+i$ ($i=2,\ldots,p-1$), which implies $x=a$. However, since $\epsilon$ is regular then it follows by Lemma\eqref{co2} that  $\max A_{1}-(a+1)=\min A_{p}-(a+p)=0$, showing that $\max A_{1}=a$ and $\min A_{p}=a+p$, as required.

\end{proof}

\begin{lemma}\label{kkk} Let $\epsilon=\left(\begin{array}{ccccc}
                                                                            A_{1} & a+2 &  \ldots & a+p-1 & A_{p} \\
                                                                            a+{1} & a+2 &  \ldots & a+p-1 & a+ p
                                                                          \end{array}
\right)$ and \\$\tau=\left(\begin{array}{ccccc}
                                                                            B_{1} & b+2 &  \ldots & b+s-1 & B_{s} \\
                                                                            b+{1} & b+2 &  \ldots & b+s-1 & b+ s
                                                                          \end{array}
\right)$ be two idempotent elements in $SReg(\mathcal{ORCP}_{n})$ with $p,s=1,2, \ldots,n$. Then $\epsilon\tau$ is  strongly regular.

\end{lemma}

\begin{proof} Let $c=\max \{a+1,b+1\}$ and $d=\min \{a+p,b+s\}$. Suppose also that $F(\epsilon)\cap F(\tau) \neq \emptyset$ and the blocks of the product $\epsilon\tau$ are $D_1, D_2, \dots, D_k$, where $k\leq \min\{p,s\}$. Then clearly, $c\leq d$. Thus we shall consider four cases:

\begin{itemize}
\item[(i)] If $a+1=c$ and $a+p=d$ then $b+1\leq a+1$ and $a+p\leq b+s$. Using convexity, it is now not difficult to see that $D_1=A_1\cup [b+1,a+1]$, $D_i=\{a+i\}~(i=2,\dots, k-1)$ and $D_k= [a+p,b+s]\cup A_p$. Moreover, $D_1\epsilon\tau=a+1=\max D_1$ and $D_k\epsilon\tau=a+p=\min D_k$. Hence $\epsilon\tau$ is a strongly regular idempotent.

The other three cases listed below are handled similarly.
\item[(ii)] If $a+1=c$ and $b+s=d$;
\item[(iii)] If $b+1=c$ and $a+p=d$;
\item[(iv)] If $b+1=c$ and $b+s=d$.

\end{itemize}
\end{proof}

Now as a consequence, from Proposition\eqref{idp1} and Lemma\eqref{kkk}, we  readily have:

\begin{theorem} Let $S=SReg(\mathcal{ORCP}_{n})$. Then  $S$ is a regular subsemigroup of $\mathcal{ORCP}_{n}$.
\end{theorem}
\noindent {\bf Acknowledgements.} The second named author would like to thank Bayero University and TET Fund for financial support. He would also like to thank The Petroleum Institute, Khalifa University of Science and Technology for hospitality during his 3-month research visit to the institution.

\end{document}